\documentclass[12pt]{amsart}
\usepackage{amssymb,amsmath,amsthm,amscd}
\usepackage[all]{xy}

\numberwithin{equation}{section}

\newtheoremstyle{fancy1}{10pt}{10pt}{\itshape}{12pt}{\textsc\bgroup}{.\egroup}{8pt}{
}
\newtheoremstyle{fancy2}{10pt}{10pt}{}{12pt}{\itshape}{.}{8pt}{ }

\theoremstyle{fancy1}

\newtheorem{cor}[equation]{Corollary}
\newtheorem{lem}[equation]{Lemma}
\newtheorem{prop}[equation]{Proposition}
\newtheorem{thm}[equation]{Theorem}

\newtheorem*{main*}{Result}
\newtheorem*{conjecture*}{Conjecture}
\newtheorem*{cor*}{Corollary}

\newcommand{\gal}[2]{\operatorname{Gal}\left( #1 / #2 \right)}
\newtheorem*{def*}{Definition}
\newtheorem{rem}[equation]{Remark}
\newtheorem*{rem*}{Remark}

\newtheorem{example}[equation]{Example}
\newtheorem*{example*}{Example}

\newtheorem*{examples*}{Examples}

\theoremstyle{remark}
\newtheorem*{case*}{Case}

\newtheorem*{proof1*}{Proof of Corollary \ref{prime2quadraticfields}}
\newtheorem*{proof2*}{Proof of Theorem \ref{nonsplitcase}}
\newtheorem*{proof3*}{Proof of Theorem \ref{groupstructure}}
\newtheorem*{proof4*}{Proof of Corollary \ref{meta-abelian}}

\newcommand{\cref}[1]{Corollary~\ref{#1}}

\newcommand{\Q}{{\mathbb{Q}}}

\newcommand{\Z}{{\mathbb{Z}}}

\def\con#1=#2(#3){#1 \equiv #2 \bmod{#3}}

\begin{document}

\title[Ray Class Groups of Quadratic and cyclotomic fields]{Ray Class Groups of quadratic and cyclotomic fields}

\author{Jing Long Hoelscher}
\address{University of Arizona  \\
     Tucson, AZ 85721}
\email{jlong@math.arizona.edu}

\begin{abstract}
This paper studies Galois extensions over real quadratic number fields or cyclotomic number fields ramified only at one prime. In both cases, the ray class groups are computed, and they give restrictions on the finite groups that can occur as such Galois groups. Also the explicit structure of ray class groups of regular cyclotomic number field is given.
\end{abstract}

\maketitle

\vspace{-0.8em}
\section{Introduction}
Let $K$ be either a real quadratic number field or a cyclotomic number field. We study the Galois extensions over $K$ ramified only at one prime in $K$ by computing the ray class groups of $K$.

In Section $2$, we consider extensions $L$ of real quadratic fields $K=\Q(\sqrt{d})$ ramified only at one prime $\mathfrak{p}$ of $K$. Say $\mathfrak{p}|p$, a rational prime. If $p$ remains prime in $K$, then there exist such extensions $L$ of arbitrarily large degree over $K$ (see Theorem \ref{nonsplitcase}). But if $p$ splits, then $[L:K]$ must be finite, and can be bounded explicitly in terms of the field $\Q(\sqrt{d})$ and the prime $p$. We first compute the ray class numbers in each case, then give restrictions on the finite groups that can occur as Galois groups over $K$ ramified only at $\frak{p}$. In the splitting case, the main results are
\begin{thm}\label{quadraticfields}
Let $K=\Q(\sqrt{d})$ be a real quadratic field with square-free positive integer $d$. Assume the prime $p$ splits completely as $\frak{p}_1\frak{p}_2$ in $K/\Q$, and $p$ does not divide the big class number of $K$. Let $L/K$ be any pro-$p$ extension ramified only at $\mathfrak{p}_1$ with Galois group $G=\gal{L}{K}$. Then $G$ is finite cyclic. 
\end{thm}
In fact, for any abelian extension $L/K$ ramified only at $\frak{p}_1$, the order $[L:K]$ is bounded explicitly in terms of $K$ and the prime $p$ (see Proposition \ref{finitecondition}). In the case the prime $p=2$ splits in $K/\Q$, any pro-nilpotent extension of $K$, that is ramified only at a single prime above $2$ in $K$, is a finite cyclic $2$-group on certain conditions of $\Q(\sqrt{d})$. 
\begin{cor}\label{prime2quadraticfields}
Let $K=\Q(\sqrt{d})$ be a real quadratic field, where $d$ is a square-free positive integer such that $d\equiv 1$ mod $8$ (so $2$ splits as $2=\mathfrak{p}_1\cdot\mathfrak{p}_2$ in $K$). Assume the big class number $|Cl_K^1|$ of $K$ is $1$. Then the maximal pro-nilpotent quotient $\pi_1^{nilp}(U_{K,\mathfrak{p}_1})$ of $\pi_1(U_{K,\mathfrak{p}_1})$ is $\Z/2^{m-1}\Z$ with $m=\nu_{\frak{p}_1}(u^2-1)-[\frac{N(u)+1}{2}]$, where $u$ is the fundamental unit of $K$ and $\nu_{\frak{p}_1}$ is the valuation at $\frak{p}_1$. Thus if $L$ is any pro-nilpotent Galois extension of $K$ ramified only at $\mathfrak{p}_1$, then $\gal{L}{K}$ is a finite cyclic $2$-group.
\end{cor}
On the other hand, in the non-split case, as the next theorem shows, $\gal{L}{K}$ can be arbitrarily large for extensions $L/K$ ramified only at one prime $p$. Any such extension $L/K$ is contained in a ray class field $R_k$ for some positive integer $k$. If we assume the fundamental unit of $K$ has norm $1$, Kitaoka constructed a maximal real sub-extension of the ray class field $R_1$ of prime conductor $p$ in \cite{Ki}. Theorem \ref{nonsplitcase} below gives a description of the ray class fields $R_k$, which relies on the explicit order of the ray class group $|Cl_K^{p^k}|=[R_k:K]$, computed in Proposition \ref{rayclassnumber}. 
\begin{thm}\label{nonsplitcase}
Let $K=\Q(\sqrt{d})$ be a real quadratic field with square-free positive integer $d$, and suppose $p$ remains prime in $K/\Q$. Let $H$, resp.\ $R_k$, be the big Hilbert class field of $K$, resp.\ the ray class field of $K$ mod $p^k$, and let $u=\frac{a+b\sqrt{d}}{2}$ be the fundamental unit of $K$.
\begin{enumerate}
\item[a)]Then for $k>\!\!>0$, $R_{k+1}=R_{k}(\zeta_{p^{k+1}})$;
\item[b)]In the case $p=2$, we also assume $N(u)=-1$ and $a,b$ are both odd. Then $R_k=H(\sqrt{2u},\zeta_{2^k})$ if $k\geq 2$, and $R_1=H$.
\end{enumerate}
\end{thm}
In section $3$, we consider the $p$-part of the ray class group $Cl_K^{\mathfrak{m}}$ of a cyclotomic number field $K=\Q(\zeta_p)$ mod $\mathfrak{m}=(1-\zeta_p)^k$ for an odd prime $p$. Proposition \ref{classnumber} gives the order of $Cl_K^{\mathfrak{m}}$ in the case $p$ is regular, which is a product of a power of $p$ and the class number of $K$; Proposition \ref{rank} shows the rank of $Cl_K^{\mathfrak{m}}$ is $\frac{p+1}{2}$; and Theorem \ref{groupstructure} below gives the explicit structure of the ray class groups for a regular prime $p$. 
\begin{thm}\label{groupstructure}
For $k\geq 3$ and a regular prime $p>2$, we write $k-2=k_0+m(p-1)$, where $0\leq k_0\leq p-2$ and $m=[\frac{k-2}{p-1}]$. Then the ray class group of $K=\Q(\zeta_p)$ mod $\frak{m}=(1-\zeta_p)^k$ is:
$$Cl^{\frak{m}}_K=\left\{ \begin{array}{cc}
    Cl_K\times (\Z/p^{m+1}\Z)^{[\frac{k_0}{2}]}\times (\Z/p^m\Z)^{\frac{p+1}{2}-[\frac{k_0}{2}]} & \text{ if } k_0\neq p-2, \\
    Cl_K\times (\Z/p^{m+1}\Z)^{\frac{p-1}{2}}\times(\Z/p^m\Z) & \text{ if } k_0=p-2.
    \end{array} \right.
  $$
\end{thm}
The formula given in the above theorem is consistent with the data in Remark \ref{data}, computed using \cite{PARI2}. As an application, we consider tamely ramified meta-abelian Galois extension over $\Q$ ramified only at $p$.
\begin{cor}\label{meta-abelian}
Let $K/\Q$ be a meta-abelian Galois extension only tamely ramified at $p$ and unramified everywhere else. Then $K$ is contained in the Hilbert class field of $\Q(\zeta_p)$.
\end{cor}
\smallskip

\section{Extensions over quadratic fields}
In this section, we always denote by $K$ a real quadratic number field $\Q(\sqrt{d})$ where $d$ is a square free positive integer. We consider algebraic extensions $L/K$ ramified only at one prime $\mathfrak{p}$ of $K$. Say $\mathfrak{p}|p$, a rational prime. If $p$ remains prime in $K$, then there exist such extensions $L/K$ of arbitrarily large degree (see Theorem \ref{nonsplitcase}). But if $p$ splits completely in $K/\Q$, the degree $[L:K]$ must be finite, and can be bounded in terms of the field $K$ and the prime $p$ (see Proposition \ref{finitecondition}). 

\subsection{Split Case}
Suppose the rational prime $p$ splits as $p=\frak{p}_1\frak{p}_2$ in the real quadratic field $K=\Q(\sqrt{d})$. The splitting condition is equivalent to the assertion that $d\equiv 1$ mod $8$ in the case $p=2$. Let $\mathcal{O}_K$ be the ring of integers of $K$. We will look at the fundamental unit of $K$.
\begin{lem}\label{fundamentalunits}
Let $K=\Q(\sqrt{d})$ be a real quadratic field with fundamental unit $u$. Assume $2=\frak{p}_1\frak{p}_2$ splits completely in $K/\Q$. Then the $\frak{p}_1$-valuation $\nu_{\frak{p}_1}(u^2-1)\geq 2$.
\end{lem}
\begin{proof}
The prime $2$ splits completely in $K/\Q$, so $d\equiv 1 \,\,(\!\!\!\!\mod 8)$. And the fundamental unit $u$ is of the form $\frac{a+b\sqrt{d}}{2}$, where $a$ and $b$ have the same parity. Since $N(u)=\pm 1$, i.e.\ $a^2-b^2d=4N(u)=\pm 4$, we have $a^2-b^2\equiv a^2-b^2d \equiv 4 \mod 8$. We know that the quadratic residues mod $8$ are $0$, $1$ and $4$, so $a$ and $b$ are both even. Write $a=2a'$ and $b=2b'$. Then $a'^2-b'^2d=\pm 1$, so $a'$ and $b'$ have different parities. Now $u=1+2(\frac{a'-1+b'\sqrt{d}}{2})\equiv 1$ mod $2$, where $\frac{a'-1+b'\sqrt{d}}{2}\in\mathcal{O}_K$. So $\nu_{\frak{p}_1}(u^2-1)=\nu_{\frak{p}_1}(4(\frac{a'-1+b'\sqrt{d}}{2})(\frac{a'-1+b'\sqrt{d}}{2}+1))\geq 2$. 
\end{proof}
For any integer $k\geq 1$, denote by $Cl_K^{\mathfrak{m}}$ the ray class group of $K$ modulo $\frak{m}=\mathfrak{p}_1^k$. Let $\mathcal{O}^*_+$, resp. $\mathcal{O}_+^{\frak{m}}$, be the group of totally positive units of $\mathcal{O}_K$, resp. of totally positive units $\equiv 1$ mod $\frak{m}$. We consider any abelian extension $L$ over $K$ ramified only at $\mathfrak{p}_1$. Define $s_{d,p}$ to be the order of $\mathcal{O}_+^*/\mathcal{O}^{\mathfrak{p}_1}_+$, so $s_{d,p}$ is the smallest integer such that $u_0^{s_{d,p}}\equiv 1$ mod $\frak{p}_1$. Also define $m_{d,p}=\nu_{\frak{p}_1}(u_0^{s_{d,p}}-1)$ to be the $\frak{p}_1$-valuation of $u_0^{s_{d,p}}-1$, so $m_{d,p}\geq 1$. The proposition below says that $\gal{L}{K}$ is finite and is bounded explicitly in terms of $s_{d,p}$, $m_{d,p}$ and the big class number.
\begin{prop}\label{finitecondition}
Suppose $K=\Q(\sqrt{d})$ is a real quadratic field with $d$ square free. Assume a prime $p=\frak{p}_1\frak{p}_2$ splits completely in $K/\Q$. Let $L$ be any abelian extension of $K$ ramified only at $\frak{p}_1$, then $\gal{L}{K}$ is finite. 
\begin{enumerate}
\item[a)]
If we also assume the fundamental unit $u$ of $K$ has norm $-1$ when $p=2$, then for any prime $p$ the ray class number 
$$|Cl^{\frak{p}_1^k}_K|=\left\{ \begin{array}{cc}
   \frac{(p-1)p^{m_{d,p}-1}}{s_{d,p}}\cdot |Cl^1_K| & \text{ if } k\geq m_{d,p}, \\
    \frac{p^{k-1}(p-1)}{s_{d,p}}\cdot|Cl_K^1| & \text{ if } k\leq m_{d,p}.
    \end{array} \right.
  $$
\item[b)]In the case $p=2$ and $N(u)=1$, the ray class number $|Cl_K^{\frak{p}_1^k}|\leq 2^{\nu_{\frak{p}_1}(u^2-1)-2}|Cl_K^1|$. 
\end{enumerate}
\end{prop}
\begin{proof}
By class field theory, we have the following exact sequence:
\begin{equation}\label{exactsequence}
1\rightarrow\mathcal{O}^*_+/\mathcal{O}^{\mathfrak{p}_1^k}_+\rightarrow(\mathcal{O}_K/\mathfrak{p}_1^k)^*\rightarrow Cl^{\mathfrak{p}_1^k}_K\rightarrow Cl^{1}_K\rightarrow 1.
\end{equation}
Proposition $1.4$ in \cite{CDO} gives the order of $(\mathcal{O}_K/\mathfrak{p}_1^k)^*$, which is $(p^f-1)\cdot p^{f(k-1)}=(p-1)p^{k-1}$, since $p$ splits completely. By the Dirichlet unit theorem, we know $\mathcal{O}_+^*$ has only one generator, say $u_0$, i.e.\ $\mathcal{O}_+^*=\{u_0^n\}_{n\in\Z}$. For part a), the valuation $\nu_{\frak{p}_1}((u_0^{s_{d,p}})^p-1)=m_{d,p}+1$: when $p>2$, $\nu_{\frak{p}_1}((u_0^{s_{d,p}})^p-1)=\min{\{pm_{d,p},m_{d,p}+1\}}=m_{d,p}+1$; when $p=2$, the condition $N(u)=-1$ and Lemma \ref{fundamentalunits} implies $m_{d,2}\geq 2$, thus also $\nu_{\frak{p}_1}((u_0^{s_{d,p}})^p-1)=m_{d,p}+1$. Taking successive powers of $p$, for any $k\geq m_{d,p}$, one has $\nu_{\frak{p}_1}((u_0^{s_{d,p}})^{p^{k-m_{d,p}}}-1)=k$ by induction on $k$. So for $k\geq m_{d,p}$, we have $\mathcal{O}^{\frak{p}_1^k}_+=\{u_0^{ns_{d,p}p^{k-m_{d,p}}}\}_{n\in\Z}$ and $\mathcal{O}^*_+/\mathcal{O}^{\frak{p}_1^k}_+\cong\Z/s_{d,p}p^{k-m_{d,p}}$. From the exact sequence \ref{exactsequence}, we have $|Cl^{\mathfrak{p}_1^k}_K|=|Cl^{\mathfrak{p}_1^{m_{d,p}}}_K|=\frac{(p-1)p^{m_{d,p}-1}|Cl_K^1|}{s_{d,p}}$ for $k\geq m_{d,p}$. For $1\leq k\leq m_{d,p}$, we have $|Cl^{\mathfrak{p}_1^k}_K|=\frac{(p-1)p^{k-1}|Cl_K^1|}{s_{d,p}}$. 

For part b), we have $\mathcal{O}^*=\{u^n\}_{n\in\Z}$ since $N(u)=1$. From lemma \ref{fundamentalunits}, we have $\nu_{\frak{p}_1}(u^2-1)\geq 2$. Then for any $k\geq \nu_{\frak{p}_1}(u^2-1)$, by induction we have $\nu_{\frak{p}_1}(u^{2^{k+1-\nu_{\frak{p}_1}(u^2-1)}}-1)=k$. So $\mathcal{O}^{\frak{p}_1^{k}}=\{u^{2^{k+1-\nu_{\frak{p}_1}(u^2-1)}}\}_{n\in\Z}$, thus $\mathcal{O}^*/\mathcal{O}^{\frak{p}_1^{k}}=\Z/(2^{k+1-\nu_{\frak{p}_1}(u^2-1)})$ for $k\geq \nu_{\frak{p}_1}(u^2-1)$. By the exact sequence \ref{exactsequence}, we have $|Cl_K^{\frak{p}_1^k}|=2^{\nu_{\frak{p}_1}(u^2-1)-2}|Cl_K^1|$ for $k\geq \nu_{\frak{p}_1}(u^2-1)$. When $k\leq\nu_{\frak{p}_1}(u^2-1)$, we have the surjective map $Cl_K^{\frak{p}_1^{\nu_{\frak{p}_1}(u^2-1)}}\twoheadrightarrow Cl_K^{\frak{p}_1^k}$, so for any $k>0$, $|Cl_K^{\frak{p}_1^k}|\leq |Cl_K^{\frak{p}_1^{\nu_{\frak{p}_1}(u^2-1)}}|=2^{\nu_{\frak{p}_1}(u^2-1)-2}|Cl_K^1|$, which concludes part b). 

Any abelian extension $L/K$ ramified only at $\mathfrak{p}_1$ is contained in the ray class field for the modulus $\mathfrak{p}_1^k$ with some positive integer $k$.
So $[L:K]$ is bounded by part a) and b).
\end{proof}
\begin{cor}Suppose $K=\Q(\sqrt{d})$ is a real quadratic field with $|Cl_K^1|=1$, where $d$ is a square-free positive integer such that $d\equiv 1$ mod $8$ (so $2$ splits as $2=\frak{p}_1\frak{p}_2$ in $K$). Define $m=\nu_{\frak{p}_1}(u^2-1)-[\frac{N(u)+1}{2}]$, where $u$ is the fundamental unit of $K$. Then
\begin{enumerate}
\item[1)] any abelian extension $L/K$ ramified only at $\frak{p}_1$ is a finite $2$-extension with degree $[L:K]\leq 2^{m-1}$;
\item[2)] every group in $\pi_A(U_{\Q(\sqrt{d}),\mathfrak{p}_1})$ is a quasi-$2$ group, where $\pi_A(U_{\Q(\sqrt{d}),\,\,\mathfrak{p}_1})$ is the set of all Galois groups of finite Galois extensions $L/K$ ramified only at $\frak{p}_1$. 
\end{enumerate}
\end{cor}
\begin{proof}
For part $1)$, if $N(u)=-1$, then $m=\nu_{\frak{p}_1}(u^2-1)=m_{d,2}$ and $u_0=u^2$. From Lemma \ref{fundamentalunits}, we know $\nu_{\frak{p}_1}(u_0-1)\geq 2$, thus $s_{d,2}=1$. We get $|Cl_K^{\frak{p}_1^k}|\leq 2^{m-1}$ by applying part a) of Proposition \ref{finitecondition}. If $N(u)=1$, $m=\nu_{\frak{p}_1}(u^2-1)-1$ and we also have $|Cl_K^{\frak{p}_1^k}|\leq 2^{m-1}$ by part b) of Proposition \ref{finitecondition}. Since $Cl_K^1$ is trivial, we have that $\gal{L}{K}$ is a $2$-group, and $[L:K]\leq 2^{m-1}$. This proves part $1)$. 

For part $2)$, take any group $G\in \pi_A(U_{\Q(\sqrt{d}),\mathfrak{p}_1})$. Denote by $p(G)$ the subgroup of $G$ generated by all elements of order a power of $2$. Then $G/p(G)$ is of odd order, thus solvable. If $G/p(G)$ is nontrivial, then there exists a nontrivial abelian odd extension of $K$ ramified only at $\mathfrak{p}_1$. This is impossible, since part $1)$ says that any such abelian extension has degree a power of $2$. So $G/p(G)$ is trivial, i.e.\ $G$ is a quasi-$2$ group.
\end{proof}
\begin{example}
In the quadratic field $K=\Q(\sqrt{17})$, the prime $2$ splits into $\mathfrak{p}_1=(\frac{5+\sqrt{17}}{2})$ and $\mathfrak{p}_2=(\frac{5-\sqrt{17}}{2})$. Here the fundamental unit $u=4+\sqrt{17}$ has norm $N(u)=-1$, and $m=\nu_{\mathfrak{p}_1}(u^2-1)=3$. The ray class field of $K$ mod $\mathfrak{p}_1^2$ is the field $\Q(\sqrt{17},\sqrt{4+\sqrt{17}})$.  The maximal abelian extension of $K$ ramified only at $\mathfrak{p}_1$ is the ray class field of $K$ mod $\mathfrak{p}_1^k$ for any $k\geq 3$, which has degree $4$ over $K$.
\end{example}
In the next proposition we will show that any $p$-extension over $K$ with big class number not divisible by $p$, and that is ramified only at $\mathfrak{p}_1$, has to be cyclic.
\begin{prop}\label{cycliccondition}
Let $L$ be a pro-$p$ extension of a quadratic number field $K=\Q(\sqrt{d})$ with the big class number not divisible by $p$. Also assume that $p$ splits completely in $K/\Q$, i.e.\ $\frak{p}=\mathfrak{p}_1\frak{p}_2$, and $L/K$ is ramified only at $\frak{p}_1$. Then $\gal{L}{K}$ is cyclic. 
\end{prop}
\begin{proof}
Denote by $G$ the Galois group $\gal{L}{K}$. Take the fixed subfield $K_0$ of the commutator subgroup $G^{'}$; then $K_0$ is the maximal abelian sub-extension of $L/K$. If $\gal{K_0}{K}$ is not cyclic, we can pick a $\Z/p\Z\times\Z/p\Z$ quotient of $\gal{K_0}{K}$, which corresponds to a subfield $K_1/K$. The field $K_1$ is not Galois over $\Q$, since it is ramified only over $\mathfrak{p}_1$ and unramified over $\mathfrak{p}_2$. The extension $K_1/K$ is totally ramified at $\frak{p}_1$ since the big class number is prime to $p$. Take the conjugate extension $K_2$ of $K_1$ over $\Q$, which will be totally ramified over $\mathfrak{p}_2$ over $K$ and unramified over $\mathfrak{p}_1$; so $K_1/K$ and $K_2/K$ are linearly disjoint. The compositum $K_1K_2$ is Galois over $\Q$. Denote by $H$ the Galois group $\gal{K_1K_2}{\Q}$. Then $H$ is the wreath product of $\gal{K_1}{K}\cong \Z/p\Z\times\Z/p\Z$ and $\gal{K}{\Q}\cong\Z/2\Z$, where the action of $\gal{K}{\Q}$ interchanges the two copies of $\Z/p\Z\times\Z/p\Z$ corresponding to $\gal{K_1}{K}$ and $\gal{K_2}{K}$. Thus $H\cong(\Z/p\Z\times\Z/p\Z)^2\rtimes \Z/2\Z$ is a semi-direct product. Denote by $x_1,x_2$, resp.\ $x_3,x_4$ the generators of $\gal{K_1}{K}$, resp.\ $\gal{K_2}{K}$. The involution $\bar{i}\in\gal{K}{\Q}$ interchanges $K_1$ and $K_2$, and $\bar{i}$ extends to an involution $i\in\gal{K_1K_2}{\Q}$, that sends $x_1$ to $x_3$ and sends $x_2$ to $x_4$. Take the subgroup $H_0$ of $H$ generated by $x_1x_3^{-1}$ and $x_2x_4^{-1}$; then $H_0\cong \Z/p\times\Z/p$. The action of $i$ on $H_0$ sends $x_1x_3^{-1}$ to $x_3x_1^{-1}=(x_1x_3^{-1})^{-1}$, and sends $x_2x_4^{-1}$ to $x_4x_2^{-1}=(x_2x_4^{-1})^{-1}$. Thus $H_0$ is normal in $H$. The quotient $H/H_0=\{\bar{x}_1=\bar{x}_3,\bar{x}_2=\bar{x}_4,\bar{i'}|\bar{x}_1^p=1,\bar{x}_2^p=1,\bar{i'}^2=1,\bar{x}_1\bar{x}_2=\bar{x}_2\bar{x}_1,\bar{x}_1\bar{i'}=\bar{i'}\bar{x}_1,\bar{x}_2\bar{i'}=\bar{i'}\bar{x}_2\}$ is abelian and $H/H_0\cong(\Z/p\Z)^2\times\Z/2\Z$, which corresponds to an abelian extension of $\Q$ ramified only at $p$ and possibly $\infty$. That is impossible by class field theory. So $\gal{K_0}{K}=G/G'$ is cyclic. By the Burnside basis theorem, $\gal{L}{K}$ is also cyclic. 
\end{proof}
Combining Propositions \ref{cycliccondition} and \ref{finitecondition} gives Theorem \ref{quadraticfields} in the introduction. For the case $p=2$, we can get a description of nilpotent extensions over a real quadratic field with a single ramification. Next we will give the proof of Corollary \ref{prime2quadraticfields} in the introduction. 
\begin{proof1*}Suppose $L/K$ is a maximal pro-nilpotent extension. Denote by $G$ the Galois group $\gal{L}{K}$. So $G$ is a pro-nilpotent group. We can decompose  $G$ into a direct product of pro-$p$ groups for various primes $p$: $G=\prod_{j=1}^s P_j$. For each $G_i:=\prod_{j\neq i}P_j$, we denote its fixed subfield of $L/K$ by $L_i$. For the Frattini subgroup $\Phi(P_i)$ of each $P_i$, we denote its fixed subfield of $L_i/K$ by $K_i$.

If $2\mid |P_i|$, then $L_i$ is a pro-$2$ extension over $K$ ramified only at $\mathfrak{p}_1$. By Proposition \ref{cycliccondition}, $L_i/K$ is a cyclic $2$-extension, thus abelian. By Proposition \ref{finitecondition}, the Galois group $\gal{L_i}{K}$ is a \textit{finite} cyclic $2$-group of order $2^{m-1}$, where $m=\nu_{\mathfrak{p}_1}(u^2-1)-[\frac{N(u)+1}{2}]$.

If $2\nmid |P_i|$, then $L_i$ is a pro-$p$ extension over $K$, where $p\mid |P_i|$ is different from $2$. Here $K_i/K$ is abelian, so by Proposition \ref{finitecondition}, $\gal{K_i}{K}$ is a finite $2$-group, thus trivial. By the Burnside basis theorem $P_i\cong \gal{L_i}{K}$ is also trivial.

So the Galois group $\gal{L}{K}$ is a finite cyclic $2$-group of order $2^{m-1}$. \hfill $\square$
\end{proof1*}
\smallskip

\subsection{Non-Split Case}
Now assume the prime $p$ remains inert in the real quadratic field $K=\Q(\sqrt{d})$. Unlike in the split case where the ray class number is bounded in terms of $d$ and $p$, here $Cl_K^{p^k}$ can be arbitrarily large when $k$ increases. 
We will first give some notations. For $k\geq 0$, denote by $Cl_K^{p^k}$ the ray class group of $K$ modulo $\frak{m}=p^k$. Let $\mathcal{O}^*_+$, resp.\ $\mathcal{O}_+^{\frak{m}}$, be the group of totally positive units of the ring of integers $\mathcal{O}_K$, resp.\ of totally positive units $\equiv 1$ mod $\frak{m}$. Let $u_0$ be the generator of $\mathcal{O}^*_+$, so $u_0$ is the fundamental unit $u$ if the norm $N(u)=1$, and $u_0=u^2$ otherwise. Define $s_{d,p}$ to be the smallest integer such that $u_0^{s_{d,p}}\equiv 1$ mod $p$, i.e.\ $s_{d,p}=|\mathcal{O}^*_+/\mathcal{O}^{p}_+|$. And define $m_{d,p}$ to be the valuation $\nu_p(u_0^{s_{d,p}}-1)$, so $m_{d,p}\geq 1$. 
\begin{lem}\label{units}
Let $K=\Q(\sqrt{d})$ be a real quadratic field with a square-free integer $d\equiv 5$ mod $8$ (so $2$ remains prime in $K$). Assume the fundamental unit $u=\frac{a+b\sqrt{d}}{2}$ of $K$ has norm $-1$. Then $m_{d,2}=2$. And $s_{d,2}=3$ if $a$ is odd; $s_{d,2}=1$ if $a$ is even.
\end{lem}
\begin{proof}
Pick any $w=x+y\sqrt{d}$ with $N(w)=-1$ and $x,y\in\Z$. We claim that $\nu_2(w^2-1)=2$. We have $x^2-y^2d=N(w)=-1$. Working modulo $8$, we know $x$, $y$ are of the form $x=8x'\pm2$, $y=2y'+1$. Write $t=\frac{x-1-y\sqrt{d}}{2}$; then $N(t)=-x/2$ and $N(t+1)=x/2$, so $\nu_p(t)=0=\nu_p(t+1)$. Thus $w^2-1=4t(1+t)$ has $2$-valuation exactly $2$. Since $N(u)=-1$, $u_0=u^2$ and $\mathcal{O}^*_+=\{u_0^{n}\}_{n\in\Z}=\{u^{2n}\}_{n\in\Z}$. If $2|a$, then by the argument above with replacing $w$ by $u$, we have $m_{d,2}=\nu_2(u_0-1)=2$, thus $s_{d,2}=1$. If $2\nmid a$, $u_0-1=ua\neq 0$ mod $2$ and $u_0^2-1=au(2+ua)\neq 0$ mod $2$. By $10.2.E$ of \cite{Ri}, $u_0^3\in\Z[\sqrt{d}]$ with $N(u_0^3)=-1$. By the same argument above with replacing $w$ by $u^3$, we know $\nu_2(u_0^3-1)=2$, thus $s_{d,2}=3$ and $m_{d,2}=2$.
\end{proof}

\begin{prop}\label{rayclassnumber}
Let $K=\Q(\sqrt{d})$ be a real quadratic field with square-free positive integer $d$. Assume the prime $p$ remains prime in $K/\Q$. 
\begin{enumerate}
\item[1)]If we also assume the fundamental unit $u$ of $K$ has norm $N(u)=-1$ when $p=2$, then for any prime $p$ the ray class number of $K=\Q(\sqrt{d})$ mod $p^k$ is:
$$|Cl^{p^k}_K|=\left\{ \begin{array}{cc}
   \frac{(p^2-1)p^{k-2+m_{d,p}}}{s_{d,p}}\cdot |Cl^1_K| & \text{ if } k\geq m_{d,p}, \\
    \frac{p^{2k-2}(p^2-1)}{s_{d,p}}\cdot|Cl_K^1| & \text{ if } k\leq m_{d,p}.
    \end{array} \right.
  $$
\item[2)]In the case $p=2$ and $N(u)=1$, we also have $|Cl_K^{2^{k+1}}|=2|Cl_K^{2^k}|$ for any $k\geq \nu_2(u^6-1)$.  
\end{enumerate}  
\end{prop}
\begin{proof}
We have $\nu_p(u_0^{s_{d,p}}-1)=m_{d,p}$ and $\mathcal{O}^*_+=\{u_0^{s_{d,p}n}\}_{n\in\Z}$. For $k\leq m_{d,p}$, by definition we know $\mathcal{O}^{p^k}_+=\{u_0^{s_{d,p}n}\}_{n\in\Z}$, thus $\mathcal{O}^*_+/\mathcal{O}^{p^k}_+\cong\Z/{s_{d,p}}$. For $k\geq m_{d,p}$: if $p>2$, directly we have the valuation $\nu_p(u_0^{s_{d,p}p^{k-m_{d,p}}}-1)=\nu_p(u_0^{s_{d,p}}-1)+k-m_{d,p}=k$; if $p=2$, by Lemma \ref{units} $m_{d,2}=\nu_p(u_0-1)=\nu_p(u^2-1)\geq 2$, so we also have $\nu_p(u_0^{s_{d,p}p^{k-m_{d,p}}}-1)=k$. By induction on $k$, we have $\mathcal{O}^{p^k}_+=\{u_0^{s_{d,p}p^{k-m_{d,p}}n}\}_{n\in\Z}$ for $k\geq m_{d,p}$, thus $\mathcal{O}^*_+/\mathcal{O}^{p^k}_+\cong \Z/p^{k-m_{d,p}}\times \Z/s_{d,p}$. Recall the exact sequence from class field theory, 
\begin{equation}\label{exactsequencequadratic}
1\rightarrow\mathcal{O}^*_+/\mathcal{O}^{\frak{m}}_+\rightarrow(\mathcal{O}_K/\frak{m})^*\rightarrow Cl^{\frak{m}}_K\rightarrow Cl^{1}_K\rightarrow 1.
\end{equation}
Together with the order $|(\mathcal{O}_K/p^k)^*|=(p^f-1)p^{f(k-1)}=(p^2-1)p^{2(k-1)}$, given by Proposition $1.4$ in \cite{CDO}, we get the conclusion in part $1)$ of the proposition.

For part $2)$, the norm $N(u)=1$, so $\mathcal{O}^*_+=\{u^n\}_{n\in\Z}$. We know $u^3\equiv 1$ mod $2$, since $|\mathcal{O}^*_+/\mathcal{O}_+^2|$ divides $|(\mathcal{O}_K/2)^*|=3$. So the valuation $\nu_2(u^6-1)\geq 2$. For any $k\geq \nu_2(u^6-1)$, we have $|\mathcal{O}_+^*/\mathcal{O}_+^{2^{k+1}}|=2|\mathcal{O}_+^*/\mathcal{O}_+^{2^{k+1}}|$.
Proposition $1.4$ in \cite{CDO} says the order $|(\mathcal{O}_K/2^k)|=3\cdot 4^{k-1}$. Combining with the exact sequence \ref{exactsequencequadratic}, we get $|Cl_K^{2^{k+1}}|=2|Cl_K^{2^k}|$ for any $k\geq \nu_2(u^6-1)$.
\end{proof}

\begin{cor}\label{rayclassnumbercor}
Let $K=\Q(\sqrt{d})$ be a real quadratic field, where $d\equiv 5$ mod $8$ is a square-free positive integer. Assume the fundamental unit $u=\frac{a+b\sqrt{d}}{2}$ of $K$ has norm $-1$ and $a,b$ are both odd. Then the ray class number of $K$ mod $2^k$ for any positive integer $k$ is as follows:
 $$|Cl^{2^k}_K|=\left\{ \begin{array}{cc}
    |Cl_K^1| & \text{ for } k=1, \\
    2^k|Cl_K^1| & \text{ for } k\geq 2.
    \end{array} \right.
  $$
Thus the ray class field of $K\!\!\mod2$ is just the big Hilbert class field of $K$.
\end{cor}
\begin{proof}
With the above assumption, Lemma \ref{units} says $s_{d,2}=3$ and $m_{d,2}=2$. Now applying Proposition \ref{rayclassnumber} to the case $p=2$ gives the ray class numbers.
\end{proof}
\begin{proof2*}Proposition \ref{rayclassnumber} says that $|Cl^{p^{k+1}}_K|=p|Cl^{p^k}_K|$ for $k\geq m_0$, where $m_0=\nu_2(u^6-1)$ if $p=2$ and $N(u)=1$, and $m_0=m_{d,p}$ otherwise. So $R_{k+1}/R_k$ is of degree $p$. Since $\Q(\zeta_{p^{k+1}})/\Q(\zeta_p)$ is of conductor $p^k$, the conductor of $R_k(\zeta_{p^{k+1}})/R_k$ divides $p^k$ and is of degree $p$, so $R_{k+1}=R_k(\zeta_{p^{k+1}})$ for $k>\!\!>0$.

For part $b)$, the ray class field $R_1$ is the big Hilbert class field by Corollary \ref{rayclassnumbercor}. If $k\geq 2$, the conductor of $K(\zeta_{2^k})/K$ divides $2^{k-2}$, and the conductor of $K(\sqrt{2u})$ over $K$ is $2^2$, so the conductor of $H(\sqrt{2u},\zeta_{2^k})/K$ divides $2^k$. Also $[H(\sqrt{2u},\zeta_{2^k}):K]=2^k|Cl_K^1|$ which equals the ray class number by Corollary \ref{rayclassnumbercor}, so $H(\sqrt{2u},\zeta_{2^k})$ is the ray class field $R_k$ of $K$ mod $p^k$, and $R_{k+1}=R_k(\zeta_{p^{k+1}})$. \hfill$\square$
\end{proof2*}
\begin{rem}
In the case $N(u)=1$, Kitaoka also constructed the maximal real subextension of $R_1/K$ in \cite{Ki}. 
\end{rem}

\smallskip

\section{Ray class groups of $\Q(\zeta_p)$}
In this section, we will consider prime cyclotomic number fields $K=\Q(\zeta_p)$ with $p>2$. Let $\mathfrak{p}$ be the unique prime $(1-\zeta_p)$ above $p$ in the ring of integer $\mathcal{O}_K$ of $K$. For any integer $k\geq 0$, denote by $Cl_K^{\mathfrak{p}^k}$ the ray class group of $K$ modulo $\frak{m}=\mathfrak{p}^k$. Let $\mathcal{O}^*_+$, resp. $\mathcal{O}_+^{\frak{m}}$, be the group of totally positive units of $\mathcal{O}_K$, resp. of totally positive units $\equiv 1$ mod $\frak{m}$. By class field theory, one has the following exact sequence: 
\begin{equation}\label{exactsequencecyclo}
1\rightarrow\mathcal{O}^*_+/\mathcal{O}^{\frak{m}}_+\rightarrow(\mathcal{O}_K/\frak{m})^*\rightarrow Cl^{\frak{m}}_K\rightarrow Cl^{1}_K\rightarrow 1.
\end{equation}
Proposition $1.4$ of \cite{CDO} gives the order of $(\mathcal{O}_K/\frak{m})^*$, which is $|(\mathcal{O}_K/\frak{m})^*|=(p-1)p^{k-1}$. 
Denote by $\mathcal{O}^*$, resp.\ $\mathcal{O}^{\mathfrak{m}}$, the group of units of $\mathcal{O}_K$, resp.\ of units $\equiv 1$ mod $\frak{m}$. Since $K=\Q(\zeta_p)$ is totally complex and all units are totally positive, we have $\mathcal{O}^*=\mathcal{O}_+^*$, $\mathcal{O}^{\frak{m}}=\mathcal{O}_+^{\frak{m}}$ and $Cl_K^1$ is the class group $Cl_K$ of $K$. The ray class number of $K$ modulo $\frak{m}$ is $|Cl^{\frak{m}}_K|=\frac{(p-1)p^{k-1}\cdot|Cl_K|}{|\mathcal{O}^*/\mathcal{O}^{\frak{m}}|}$ by the exact sequence \ref{exactsequencecyclo}. For $2\leq i\leq \frac{p-1}{2}$, let $u_i=\frac{\zeta_p^i-1}{\zeta_p-1}$ be the cyclotomic units, so $u_i=\sum_{j=1}^i(_j^i)(\zeta_p-1)^{j-1}$. Denote by $\lambda$ the primitive $(p-1)$th root of unity in $(\Z/p)^*$. 
\begin{lem}\label{k=1}
For $k=1,2$, we have $\mathcal{O}^*/\mathcal{O}^{\mathfrak{p}^k}\cong (\mathcal{O}_K/\mathfrak{p}^k)^*$, thus $Cl_K^{\mathfrak{p}}\cong Cl_K \cong Cl_K^{\mathfrak{p}^2}$.
\end{lem}
\begin{proof}
For $k=1,2$, there is a natural inclusion $i: \mathcal{O}^*/\mathcal{O}^{\mathfrak{p}^k}\hookrightarrow(\mathcal{O}_K/\mathfrak{p}^k)^*$, where $(\mathcal{O}_K/\mathfrak{p})^*\cong \mathbb{F}_p^*$ and $(\mathcal{O}_K/\mathfrak{p}^2)^*\cong\Z/p(p-1)\Z$. Consider the cyclotomic unit $u_{\lambda}$, we know $u_{\lambda} \equiv \lambda$ mod $(\zeta_p-1)$, i.e.\ $u_{\lambda}$ has order $p-1$ mod $(\zeta_p-1)$; also we have $u_{\lambda}^p\equiv \lambda$ mod $(\zeta_p-1)^2$, so $\bar{u}_{\lambda}\in\mathcal{O}^*/\mathcal{O}^{\mathfrak{p}^2}$ has order $p(p-1)$. Thus $\mathcal{O}^*/\mathcal{O}^{\mathfrak{p}^k}\cong (\mathcal{O}_K/\mathfrak{p}^k)^*$ and $Cl_K^{\mathfrak{p}}\cong Cl_K \cong Cl_K^{\mathfrak{p}^2}$ by the exact sequence \ref{exactsequencecyclo}. 
\end{proof}
\begin{lem}\label{odd}
For any integer $k\geq 2$, if $k$ is odd or $(p-1)|k$, then $\mathcal{O}^{\mathfrak{p}^k}=\mathcal{O}^{\mathfrak{p}^{k+1}}$.
\end{lem}
\begin{proof}
For any $k\geq 2$, if $x\in\mathcal{O}^{\mathfrak{p}^k}$, then $x-1\equiv 0$ mod $(\zeta_p-1)^2$. Write $x=\zeta_p^{a_x}x_0$, where $x_0\in\Q(\zeta_p+\zeta_p^{-1})$. Since $\zeta_p^{a_x}=(1+(\zeta_p-1))^{a_x}\equiv 1+a_x(\zeta_p-1)$ mod $(\zeta_p-1)^2$, we have $p|a_x$, thus $x=x_0\in\Q(\zeta_p+\zeta_p^{-1})$ is real. So $\nu_{\frak{p}}(x-1)$ is even, i.e. $\mathcal{O}^{\mathfrak{p}^k}=\mathcal{O}^{\frak{p}^{k+1}}$ for $k$ is odd. For any $k=k'(p-1)\geq (p-1)$, if $x\in\mathcal{O}^{\mathfrak{p}^k}$, then $x\equiv 1$ mod $p^{k'}$. Since every element of $\Z[\zeta_p]$ is congruent to a rational integer modulo $\mathfrak{p}$, we can write $x=1+p^{k'}a +p^{k'}y(\zeta_p-1)$, where $y\in\Z[\zeta_p]$ and $a\in \Z$. Then $1=N(x)\equiv (1+p^{k'}a)^{p-1}\equiv 1+(p-1)p^{k'}a\equiv 1-p^{k'}a$ mod $\mathfrak{p}^{k+1}$, so $(\zeta_p-1)|a$ and $x\equiv 1$ mod $\mathfrak{p}^{k+1}$; thus $\mathcal{O}^{\frak{p}^k}=\mathcal{O}^{\frak{p}^{k+1}}$. 
\end{proof}
\begin{lem}\label{even}
Assume $p\nmid |Cl_{\Q(\zeta_p+\zeta_p^{-1})}|$. Then for any even integer $2\leq k\leq p-3$, we have 
$\mathcal{O}^{\frak{p}^{k+\gamma(p-1)}}/\mathcal{O}^{\frak{p}^{k+1+\gamma(p-1)}}\cong \Z/p\Z$ if $\gamma \geq \gamma_k$; and $\mathcal{O}^{\frak{p}^{k+\gamma(p-1)}}/\mathcal{O}^{\frak{p}^{k+1+\gamma(p-1)}}=\{1\}$ if $0\leq \gamma< \gamma_k$, where $\gamma_k\geq 0$ is the smallest integer such that $p^{2\gamma_k+1}\nmid B_{kp^{\gamma_k}}$. Thus if $p$ is regular, we have $\mathcal{O}^{\mathfrak{p}^{k'}}/\mathcal{O}^{\mathfrak{p}^{k'+1}}\cong \Z/p$ for any even integer $k'$ such that $(p-1)\nmid k'$.  
\end{lem}
\begin{proof}
Let $E_{k}^{(N)}=\prod_{i=1}^{p-1}(\zeta_p^{\frac{1-\lambda}{2}}\cdot\frac{1-\zeta_p^{\lambda}}{1-\zeta_p})^{\omega_N(i)^{k}\sigma_i^{-1}}$ as in section $8.3$ of \cite{Wa}, where $\omega_N(i)\equiv \omega(i)$ mod $p^N$, $\omega$ is the Teichm\"{u}ller character and $\sigma_i(\zeta_p)=\zeta_p^i$. By Proposition $8.12$ of \cite{Wa}, we have $\nu_\frak{p}(\log_p(E_{k}^{(N)}))=k+(p-1)\nu_{p}(L_p(1,\omega^k))$ for $N\geq 1+\nu_{p}(L_p(1,\omega^k))$. By Theorems $5.11$ and $5.12$ in \cite{Wa}, for any $\gamma\geq 0$ we have $L_p(1,\omega^{kp^{\gamma}})\equiv L_p(1-kp^{\gamma},\omega^{kp^{\gamma}})=-(1-p^{kp^{\gamma}-1})\frac{B_{kp^{\gamma}}}{kp^{\gamma}}\equiv -\frac{B_{kp^{\gamma}}}{kp^{\gamma}}$ mod $p^{\gamma+1}$. So $\gamma_k=\nu_{p}(L_p(1,\omega^{k}))$. Write $E_{k}^{(1+\gamma_k)}=a_{k}+b_{k}(\zeta_p-1)^{k+\gamma_k(p-1)}$ mod $\mathfrak{p}^{k+2+\gamma_k(p-1)}$ with $p\nmid a_{k}b_{k}$. Let $1\leq a_{k}^{-1}\leq q-1$ be the inverse of $a_{k}$ in $(\Z/p)^*$; then $u_{a_{k}^{-1}}^p\equiv a_{k}^{-1}$ mod $p$. Define $x_{k}=E_{k}^{(1+\gamma_k)}\cdot u_{a_{k}^{-1}}^p$; then $x_{k}\equiv 1+a_{k}^{-1}b_{k}(\zeta_p-1)^{k+\gamma_k(p-1)}$ mod $\mathfrak{p}^{k+2+\gamma_k(p-1)}$, i.e. $x_{k}\in\mathcal{O}^{\frak{p}^{k+\gamma_k(p-1)}}-\mathcal{O}^{\frak{p}^{k+1+\gamma_k(p-1)}}$. Thus $x_{k}^{\gamma-\gamma_k}\in\mathcal{O}^{\frak{p}^{k+\gamma(p-1)}}-\mathcal{O}^{\frak{p}^{k+1+\gamma(p-1)}}$ for any $\gamma\geq \gamma_k$. So $\mathcal{O}^{\frak{p}^{k+\gamma(p-1)}}\neq\mathcal{O}^{\frak{p}^{k+1+\gamma(p-1)}}$ for any $\gamma\geq\gamma_k$. On the other hand, we have a natural injection 
$\mathcal{O}^{\mathfrak{p}^{k+\gamma(p-1)}}/\mathcal{O}^{\mathfrak{p}^{k+1+\gamma(p-1)}}\hookrightarrow \frac{1+\frak{p}^{k+\gamma(p-1)}}{1+\mathfrak{p}^{k+1+\gamma(p-1)}}\cong \Z/p$. So $\frac{\mathcal{O}^{\mathfrak{p}^{k+\gamma(p-1)}}}{\mathcal{O}^{\mathfrak{p}^{k+1+\gamma(p-1)}}}\cong \Z/p$ for $\gamma\geq \gamma_k$. When $\gamma<\gamma_k$, the condition $p\nmid|Cl_{\Q(\zeta_p+\zeta_p^{-1})}|$ implies the $E_i^{(1)}$'s generate $\mathcal{O}^*/(\mathcal{O}^*)^p$ where $i$ runs through all even integers in $[2,p-3]$. So $\frac{\mathcal{O}^{\mathfrak{p}^{k+\gamma(p-1)}}}{\mathcal{O}^{\mathfrak{p}^{k+1+\gamma(p-1)}}}=\{1\}$ for $\gamma<\gamma_k$. If $p$ is regular, then $p\nmid B_k$, i.e. $\gamma_k=0$ for any even integer $2\leq k\leq p-3$, and the conclusion follows.
\end{proof}

\vspace{-0.5em}
Combining Lemmas \ref{k=1}, \ref{odd} and \ref{even}, we can give the order of the ray class groups in the case $p$ is regular.
\begin{prop}\label{classnumber}
For an odd regular prime $p$, the ray class number of $K=\Q(\zeta_p)$ mod $\mathfrak{m}=(\zeta_p-1)^k$ is:
\vspace{-1em}
$$|Cl^{\frak{m}}_K|=\left\{ \begin{array}{cc}
    |Cl_K| & \text{ if } k\leq 2, \\
    |Cl_K|\cdot p^{[\frac{k}{2}]+[\frac{k-1}{p-1}]-1} & \text{ if } k\geq 3.
    \end{array} \right.
  $$
\end{prop}
\begin{proof}
When $k\leq 2$, this is just Lemma \ref{k=1}. For $k\geq 3$, we know $|Cl_K^{\mathfrak{p}^{k+1}}/Cl_K^{\mathfrak{p}^{k}}|=p/{[\mathcal{O}^{\frak{p}^k}:\mathcal{O}^{\frak{p}^{k+1}}]}$ by the exact sequence \ref{exactsequencecyclo}. The conclusion directly comes from Lemma \ref{odd} and Lemma \ref{even}.
\end{proof}
\begin{rem}\label{divisibility}
The above proposition shows the ray class number of $K=\Q(\zeta_p)$ mod $(1-\zeta_p)^k$ for an odd regular prime $p$ and any $k\in\mathbb{N}$ is a product of a power of $p$ (possibly $1$) and $|Cl_K|$. In fact the statement is true even without the condition $p$ is regular: by Lemma \ref{k=1} we have $p-1||\mathcal{O}^*/\mathcal{O}^{\frak{p}^k}|$; combining with the exact sequence \ref{exactsequencecyclo} we get the same conclusion. 
\end{rem}
Denote by $\hat{\mathcal{O}}$ and $\hat{\mathfrak{p}}$ the valuation ring in the $p$-adic completion $\Q_p(\zeta_p)$ and its maximal ideal. Denote by $\hat{\mathcal{O}}^*$, resp.\ $\hat{U}^{(k)}$, the group of units, resp.\ the group of units $\equiv 1$ mod $\hat{\mathfrak{p}}^k$ in $\hat{\mathcal{O}}$. The next proposition will give the $p$-rank of the ray class group, i.e.\ the minimal number of generators for the $p$-part of the ray class group.
\begin{prop}\label{rank}
For any odd prime $p$ and any integer $k\gg 0$, The $p$-rank of the ray class group of $K=\mathbb{Q}(\zeta_p)$ mod $\mathfrak{m}=(\zeta_p-1)^{k}$ is $(p+1)/2$.
\end{prop}
\begin{proof}
We will show that the $p$-part of the ray class group has rank $(p+1)/2$ when $k\rightarrow \infty$. Taking the projective limit of the exact sequence \ref{exactsequencecyclo}, we get
$$1\rightarrow \lim_{\leftarrow}\mathcal{O}^*/\mathcal{O}^{\frak{m}}\rightarrow \lim_{\leftarrow}(\mathcal{O}_K/\frak{m})^*\rightarrow \lim_{\leftarrow} Cl^{\frak{m}}_K\rightarrow Cl^{1}_K\rightarrow 1,\hspace{2em} (*),$$
On the one hand, we have 
$$\lim_{\leftarrow} (\mathcal{O}_K/\mathfrak{p}^k)^*\cong \lim_{\leftarrow}(\hat{\mathcal{O}}/\hat{\mathfrak{p}}^k)^*\cong \lim_{\leftarrow}\hat{\mathcal{O}}^*/\hat{U}^{(k)}\cong \hat{\mathcal{O}}^*=(\Z_p(\zeta_p))^*.$$ 
In the discrete valuation ring, the units $(\Z_p(\zeta_p))^*=\Z/(p-1)\Z\times \Z/p\Z\times U^{(2)}$. When $k=2$, we have the topological isomorphism of $\Z_p$-modules, $\log: \hat{U}^{k}\longrightarrow \hat{\mathfrak{p}}^k\cong\hat{\mathcal{O}}.$
Since $\hat{\mathcal{O}}$ admits an integral basis $\alpha_1,\dots, \alpha_{p-1}$, i.e.\ $\hat{\mathcal{O}}\cong\Z_p\alpha_1\oplus\cdots\oplus\Z_p\alpha_{p-1}\cong\Z_p^{p-1}$, we have $\hat{U}^{(2)}\cong \Z_p^{p-1}$. Thus the projective limit $\lim_{\leftarrow}(\mathcal{O}_K/\mathfrak{m})^*\cong \Z/(p-1)\Z\oplus\Z/p\Z\oplus\Z_p^{p-1}$.
On the other hand, $\lim_{\leftarrow} \mathcal{O}^*/\mathcal{O}^\mathfrak{m}=\lim_{\leftarrow}\mathcal{O}^*/\mathcal{O}^\mathfrak{p}\times\lim_{\leftarrow}\mathcal{O}^{\mathfrak{p}}/\mathcal{O}^\mathfrak{m}$. By Lemma \ref{k=1}, we have $\lim_{\leftarrow}\mathcal{O}^*/\mathcal{O}^{\mathfrak{p}}\cong\Z/(p-1)\Z$, and $\lim_{\leftarrow}\mathcal{O}^{\mathfrak{p}}/\mathcal{O}^\mathfrak{m}$ is the pro-$p$-completion $\mathcal{O}^{\mathfrak{p}}\otimes\Z_p$. Since Leopoldt's conjecture is true for cyclotomic fields (see Corollary $5.32$ of \cite{Wa}), we know that the $\Z_p$-rank of the closure of the global principal units $\mathcal{O}^{\mathfrak{p}}$ in the local units $(\Z_p[\zeta_p])^*$ is $\frac{p-3}{2}$. Together with the $p$-th roots of unity in $\mathcal{O}^{\mathfrak{p}}$, we have $$\lim_{\leftarrow} \mathcal{O}^*/\mathcal{O}^{\mathfrak{m}}\cong \Z/(p-1)\Z\times\Z/p\Z\times \Z_p^{\frac{p-3}{2}}.$$ 
Therefore from the exact sequence $(*)$, we get $Cl_K^{\mathfrak{m}}\cong Cl_K\times \Z_p^{(p+1)/2}$.
\end{proof}
We now focus on the explicit structure of the ray class groups.
\begin{lem}\label{groupstructureinitial}
Assume $p$ is a regular odd prime. For $3\leq k\leq p$, the ray class group $Cl_K^{\mathfrak{p}^k}\cong Cl_K\times (\Z/p\Z)^{[\frac{k}{2}]-1}$; and for $k=p$, the ray class group $Cl_K^{\mathfrak{p}^p}\cong (\Z/p\Z)^{\frac{p-1}{2}}$.
\end{lem}
\begin{proof}
{F}rom the exact sequence \ref{exactsequencecyclo}, we have $Cl_K^{\mathfrak{m}}\cong \frac{(\mathcal{O}/\frak{m})^*}{\mathcal{O}^*/\mathcal{O}^{\frak{m}}}\times Cl_K$, where $\frak{m}=\frak{p}^k$. So the $p$-part of the ray class group is a quotient of $(\mathcal{O}/\frak{m})^*$. We have $(\mathcal{O}/\frak{m})^*\cong (\hat{\mathcal{O}}/\hat{\frak{m}})^*\cong \hat{\mathcal{O}}^*/\hat{U}^{(k)}\cong\Z/(p-1)\Z\times\Z/p\Z\times\hat{U}^{(2)}/\hat{U}^{(k)}$. For $k\leq p$, we know $\hat{U}^{(2)}/\hat{U}^{(k)}\cong(\Z/p\Z)^{k-2}$. So $(\mathcal{O}/\frak{m})^*\cong \Z/(p-1)\Z\times(\Z/p\Z)^{k-1}$. Thus Proposition \ref{classnumber} gives the conclusion.
\end{proof}
Next we will use Propositions \ref{classnumber}, \ref{rank} and Lemma \ref{groupstructureinitial} to prove Theorem \ref{groupstructure} in the introduction. 
\begin{proof3*}
For $k\leq p$, the conclusion is just Lemma \ref{groupstructureinitial}. Now assume $k>p$. On the one hand, we know that $(\mathcal{O}_K/\mathfrak{m})^*\cong \Z/(p-1)\Z\times(\Z/p\Z)\times \hat{U}^{(2)}/\hat{U}^{(k)}\cong \Z/(p-1)\Z\times\Z/p\Z\times(\Z/p^m\Z)^{p-1-k_0}\times(\Z/p^{m+1}\Z)^{k_0}$; on the other hand, $\mathcal{O}^*/\mathcal{O}^{\mathfrak{m}}\cong \mu_p\times\Z/(p-1)\Z\times \mathcal{O}^{\mathfrak{p}^2}/\mathcal{O}^{\mathfrak{p}^k}$. Pick any $x\in\mathcal{O}^{\mathfrak{p}^2}$, define $v_x=\nu_{\mathfrak{p}}(x-1)$; so $v_x\geq 2$. Then $\nu_{\frak{p}}(x^{p^{m+1}}-1)=(m+1)(p-1)+v_x\geq k$, i.e.\ $\bar{x}$ has order dividing $p^{m+1}$ in $\mathcal{O}^{\mathfrak{p}^2}/\mathcal{O}^{\mathfrak{p}^k}$. If there exists an element $\bar{y}\in\mathcal{O}^{\mathfrak{p}^2}/\mathcal{O}^{\mathfrak{p}^k}$ which has order less than $p^m$, then $v_y+(m-1)(p-1)=\nu_{\frak{p}}(y^{p^{m-1}}-1)\geq k$, so $v_y\geq p$. Thus $y$ is a unit and $y\equiv 1$ mod $p$, by Theorem $5.36$ of \cite{Wa}, we have $y=y_0^p$ for some $y_0\in\mathcal{O}^{\frak{p}^2}$. So $\mathcal{O}^{\mathfrak{p}^2}/\mathcal{O}^{\mathfrak{p}^k}$ is a product of copies of $\Z/p^m$ and $\Z/p^{m+1}$. By Proposition \ref{classnumber}, we can also get the order of $\mathcal{O}^{\mathfrak{p}^2}/\mathcal{O}^{\mathfrak{p}^k}$, by the exact sequence \ref{exactsequencecyclo} the structure: 
$$\mathcal{O}^{\mathfrak{p}^2}/\mathcal{O}^{\mathfrak{p}^k}\cong\left\{ \begin{array}{cc}
    (\Z/p^{m+1}\Z)^{k_0-[\frac{k_0}{2}]}\times(\Z/p^m\Z)^{\frac{p-3}{2}-k_0+[\frac{k_0}{2}]} & \text{ if } k_0\neq p-2, \\
    (\Z/p^{m+1}\Z)^{\frac{p-3}{2}} & \text{ if } k_0=p-2.
    \end{array} \right.
  $$
Notice $Cl_K^{\mathfrak{p}^p}\subset Cl_K^{\mathfrak{p}^k}$, so the $p$-rank of the $p$-part of the ray class group is no less than $\frac{p+1}{2}$, thus is exactly $\frac{p+1}{2}$ by Proposition \ref{rank}, which forces the ray class groups to have the structure as in the conclusion.\hfill$\square$ 
\end{proof3*}
\begin{rem}\label{data}
Using \cite{PARI2}, we compute the ray class groups of some cyclotomic number fields as in Table \ref{rayclassgroups}, which are consistent with the propositions above. 
\end{rem}
{\setlength{\tabcolsep}{0.06cm}
\renewcommand{\arraystretch}{1.4}
\begin{table}[!h]
\vspace{-1em}
\begin{center}
\begin{tabular}{|c|c|c|c|c|}
\hline
$K$ & $\mathfrak{p}$ & $Cl_K$ & $Cl^{p}_K$ & $Cl^{p^2}_K$ \\
\hline \hline
$\Q(\zeta_{23})$ & $23\mathcal{O}_{K}$ & $\Z/3$ & $\Z/(3\!\cdot\!23)\!\!\times\!\!(\Z/23)^9$ & $\Z/(3\!\cdot\!23^2)\!\!\times\!\!(\Z/23^2)^9\!\!\times\!\!(\Z/23)^2$ \\
\hline
$\Q(\zeta_{29})$ & $29\mathcal{O}_{K}$ & $(\Z/2)^3$ & $(\Z/2\!\cdot\!29)^3\!\!\times\!\!(\Z/29)^{10}$ & $(\Z/2\!\cdot\!29^2)^3\!\!\times\!\!(\Z/29^2)^{10}\!\times\!(\Z/29)^2$ \\
\hline
$\Q(\zeta_{31})$ & $31\mathcal{O}_{K}$ & $\Z/9$ & $(\Z/9\!\cdot\!31)\!\!\times\!\!(\Z/31)^{13}$ & $(\Z/9\!\cdot\!31^2)\!\!\times\!\!(\Z/31^2)^{13}\!\!\times\!\!(\Z/31)^2$ \\
\hline
$\Q(\zeta_{37})$ &  $37\mathcal{O}_{K}$ & $\Z/37$ & $\Z/(37^2)\!\!\times\!\!(\Z/37)^{17}$ & $(\Z/37^3)\!\!\times\!\!(\Z/37^2)^{16}\!\!\times\!\!(\Z/37)^3$\\
\hline
$\Q(\zeta_{41})$ & $41\mathcal{O}_{K}$ & $(\Z/11)^2$ & $(\Z/41\!\cdot\!11)^2\!\!\times\!\!(\Z/41)^{17}$ & $(\Z/41^2\!\cdot\! 11)^2\!\!\times\!\!(\Z/41^2)^{17}\!\!\times\!(\!\Z/41)^2$ \\
\hline
$\Q(\zeta_{43})$ & $43\mathcal{O}_{K}$ & $\Z/211$ &  $(\Z/211\!\cdot\!\!43)\!\!\times\!\!(\Z/43)^{19}$ & $\Z/(211\!\cdot\!\!43^2)\!\!\times\!\!(\Z/43^2)^{19}\!\!\times\!\!(\Z/43)^2$ \\
\hline
$\Q(\zeta_{47})$ & $47\mathcal{O}_{K}$ & $\Z/(5\!\cdot\!\!139)$ & $(\Z/5\!\cdot\!\!139\!\cdot\!\!47)\!\!\times\!\!(\Z/47)^{21}$ & $\Z/(5\!\cdot\!\!139\!\cdot\!\!47^2)\!\!\times\!\!(\Z/47^2)^{21}\!\!\!\times\!\!(\Z/47)^2$\\
\hline
\end{tabular}
\caption{Ray class groups of cyclotomic number fields}\label{rayclassgroups}
\end{center}
\vspace{-1em}
\end{table}}

As an application of the ray class groups, we consider extensions $K/\Q$ with meta-abelian Galois group $G$, i.e.\ $G^{''}$ is trivial. Next we will give the proof of Corollary \ref{meta-abelian} in the introduction, which is very similar to a statement discovered independently in \cite{JP}.
\begin{proof4*}
Consider the fixed field $K'$ of $G'$ in $K/\Q$. Since $G/G'$ is abelian, and $K'$ is tamely ramified and ramified only at $p$, we know $K'$ is contained in $\Q(\zeta_p)$. Since $K/K'$ is abelian and ramified only at $p$, we know $K$ is contained in a ray class field $R$ of $\Q(\zeta_p)$ for a modulus $\mathfrak{m}=(1-\zeta_p)^k$ for some $k$. From Remark \ref{divisibility}, we know that $|Cl_{\Q(\zeta_p)}^{\frak{m}}|$ is a product of a power of $p$ and the class number of $\Q(\zeta_p)$. By class field theory $R/\Q(\zeta_p)$ is of degree $|Cl_{\Q(\zeta_p)}^{\frak{m}}|$, a product of a power of $p$ and the class number $Cl_{\Q(\zeta_p)}$. Therefore $R$ corresponds to a $p$-extension of the Hilbert class field $H$ of $\Q(\zeta_p)$. Consider any prime $\ell$ in $R$. The inertia degree $|I_{R/\Q(\zeta_p)}^{\ell}|=|I_{R/H}^{\ell}|$ of $\ell$ in $R/\Q(\zeta_p)$ divides $[R:H]$ and is a power of $p$. So $|I_{K/K'}|$ is also a power of $p$. But $K/\Q$ is tamely ramified. So $I^{\ell}_{K/(K'\cap H)}$ is trivial, thus $K/K'$ is unramified, i.e.\ $K$ is contained in the Hilbert class field of $\Q(\zeta_p)$.\hfill$\square$
\end{proof4*}

\smallskip
\providecommand{\bysame}{\leavevmode\hbox
 to3em{\hrulefill}\thinspace}

\end{document}